\documentclass{article}

\usepackage{arxiv}

\usepackage[utf8]{inputenc} 
\usepackage[T1]{fontenc}    
\usepackage{hyperref}       
\usepackage{url}            
\usepackage{booktabs}       
\usepackage{amsfonts}       
\usepackage{nicefrac}       
\usepackage{microtype}      
\usepackage{lipsum}
\usepackage{amsmath}

\usepackage{amsfonts,amssymb,stmaryrd}
\usepackage{graphicx}
\usepackage{amsthm}
\usepackage{xcolor}

\newcommand{\R}{\mathbb{R}}

\newtheorem{theorem}{Theorem}[section]
\newtheorem{proposition}{Proposition}[section]
\newtheorem{lemma}{Lemma}[section]
\newtheorem{corollary}{Corollary}[section]

\newtheorem{definition}{Definition}[section]
\newtheorem{remark}{Remark}[section]

\usepackage{subcaption}
\usepackage{cite}
\usepackage{authblk}
\newcommand{\p}{\partial}
\newcommand{\bb}{\begin{equation}}
\newcommand{\ee}{\end{equation}}
\newcommand{\ba}{\begin{array}}
\newcommand{\ea}{\end{array}}
\newcommand{\f}{\frac}
\usepackage[all]{xy}
\newcommand{\ds}{\displaystyle}

\newcommand{\al}{\alpha}
\newcommand{\be}{\beta}

\newcommand{\sign}{\text{sgn}\,}

\newcommand{\N}{{\mathbb N}}
\newcommand{\M}{{\cal M}}
\newcommand{\Span}{\text{Span}\,}

\newcommand{\supp}{\text{supp}}
\newcommand{\Diag}{\text{Diag}}

\newcommand{\letra}{ \renewcommand{\labelenumi}{\alph{enumi})}}  
\usepackage{tikz-cd}

\numberwithin{equation}{subsection}

\usepackage{ulem}

\usepackage[most]{tcolorbox}

\definecolor{TextColor}{rgb}{0.75, 0.75, 0.75}

\title{Local isometric immersions and breakdown of manifolds determined by Cauchy problems of the Degasperis-Procesi equation}

\author[1,2] {Igor Leite Freire}
\affil[1]{Department of Mathematical Sciences,
Loughborough University\\
 LE11 3TU Epinal Way, Loughborough, United Kingdom\\
\texttt{I.Leite-Freire@lboro.ac.uk}
}
\affil[2]{Departamento de Matemática, Universidade Federal de São Carlos\\
Rodovia Washington Luís, Km 235, 13565-905\\
São Carlos-SP, Brasil\\
  \texttt{igor.freire@ufscar.br} \\
  \texttt{igor.leite.freire@gmail.com}}
 
\begin{document}
\maketitle
\begin{abstract}
We prove that both local and non-local formulations of the Degasperis-Procesi equation possess a pseudospherical nature. As a result, solutions determined by Cauchy problems with non-trivial initial data and a minimal specific regularity define an orthonormal coframe for a pseudospherical metric within a designated strip. While the region is entirely described by the initial condition, the one-forms are determined by the corresponding solutions. Moreover, we establish that any non-trivial initial condition yields a second fundamental form, which is locally determined within the strip. The surface may collapse within a finite-height strip when the initial condition leads to wave breaking. We determine conditions for the coframe to be defined on the upper plane. We investigate possible integrability structures related to the triad of fundamental one-forms.
\end{abstract}

{\bf MSC classification 2020:}  58J60,  53C21, 35R01, 35A01, 35Q51.

\newpage
\tableofcontents
\newpage

\section{Introduction}\label{sec1}

In recent decades, a significant focus has been placed on a third order dispersive equation known as $b-$equation \cite{dehoho}:
\bb\label{1.0.1}
u_t-u_{txx}+(b+1)uu_x=bu_xu_{xx}+uu_{xxx},
\ee
$b\in\R$. It has attracted much attention from diverse fields such as hydrodynamics, mathematical physics and analysis of PDEs ever since, see \cite{escher2008,pri-AML,freire-AML,holm-siam,holm-pla,zhou} and references therein.

Some of the main features of this family are their peakon solutions, given by
$$
u(x,t)=ce^{-|x-ct|},
$$
$c\in\R$, and the fact that the equation itself is a conservation law, that is,
\bb\label{1.0.2}
\p_t(u)+\p_x\Big(\f{b+1}{2}u^2-u_x^2-uu_{xx}-u_{tx}\Big)=0,
\ee
see \cite{holm-siam,holm-pla}. This last property seems to be, indeed, the only conservation law admitted by certain members, see \cite{holm-siam,holm-pla,freirejde}.

Despite the properties above, no one would dispute that most of the relevance of \eqref{1.0.1} comes from the fact that it encloses two quite famous integrable equations, namely, the Camassa-Holm (CH) \cite{chprl}
\bb\label{1.0.3}
u_t-u_{txx}+3uu_x=2u_xu_{xx}+uu_{xxx},
\ee
corresponding to $b=2$ in \eqref{1.0.1}, and the member $b=3$,
\bb\label{1.0.4}
u_t-u_{txx}+4uu_x=3u_xu_{xx}+uu_{xxx},
\ee
discovered by Degasperis and Procesi (DP) \cite{deg} and named after them.

Although interpreting the CH and DP equations as the two most renowned siblings within the family of $b-$equations is both natural and pertinent, they cannot be regarded as identical twins due to their distinct properties. Several differences, particularly from the perspective of the analysis of partial differential equations, are pointed out in \cite{liu2006,liu2007}.

The only two integrable members of \eqref{1.0.1} are just the CH and the DP equations, see \cite[Theorem 4]{mik}. In particular, the CH equation has a second order Lax pair \cite[Equation (6)]{chprl}, given by
\bb\label{1.0.5}
\ba{lcl}
\psi_{xx}&=&\ds{\Big(\f{1}{4}-\f{m}{2\lambda}\Big)\psi},\\
\\
\psi_t&=&\ds{-(\lambda+u)\psi_x+\f{1}{2}u_x\psi},
\ea
\ee
whereas the DP equation admits a third order Lax pair \cite[Equation (3.10)]{dehoho}
\bb\label{1.0.6}
\ba{lcl}
\psi_{xxx}&=&\psi_x-\lambda m\psi,\\
\\
\psi_t&=&\ds{\Big(u_x+\f{2}{3\lambda}\Big)\psi-\f{1}{\lambda}\psi_{xx}-u\psi_x}.
\ea
\ee

As a result of these Lax pairs, their corresponding zero curvature representations (ZCR) 
\bb\label{1.0.7}
\p_tX-\p_xT+[X,T]=0
\ee
are significantly different, since for the CH equation we have a $\frak{sl}(2,\R)-$valued representation, whereas its equivalent for the DP equation is $\frak{sl}(3,\R)-$valued.

The $\frak{sl}(2,\R)$ ZCR for the CH equation reveals an intrinsic geometric structure behind its solutions. In fact, applying the ideas coming from \cite{sasaki,chern}, under certain weak conditions its solutions give rise to certain two-dimensional abstract manifolds, more precisely, a pseudospherical surface (PSS – the same acronym will be used for both singular and plural forms) with Gaussian curvature ${\cal K}=-1$. For more details about ZCR and PSS equations, see \cite{reyes2000}.

A two-dimensional manifold ${\cal M}$ is said to be a PSS if there exists one-forms $\omega_1,\,\omega_2,\,\omega_3$ on ${\cal M}$ such that $\omega_1\wedge\omega_2\neq0$ and
\bb\label{1.0.8}
d\omega_1=\omega_3\wedge\omega_2,\quad d\omega_2=\omega_1\wedge\omega_3,\quad d\omega_3=\omega_1\wedge\omega_2.
\ee
If such forms exist, then they endow ${\cal M}$ with the metric 
\bb\label{1.0.9}
g=\omega_1^2+\omega_2^2
\ee
and a Gaussian curvature ${\cal K}=-1$ \cite{chern,reyes2011}.

As a result of \eqref{1.0.8}, if we write $\omega_{i}=f_{i1}dx+f_{i2}dt$, $1\leq  i\leq 3$, and define
\bb\label{1.0.10}
X=\f{1}{2}\begin{pmatrix}
f_{21}& f_{11}-f_{31}\\
\\
f_{11}+f_{31} & -f_{21}
\end{pmatrix}\quad\text{and}\quad
T=\f{1}{2}\begin{pmatrix}
f_{22}& f_{12}-f_{32}\\
\\
f_{12}+f_{32} & -f_{22}
\end{pmatrix},
\ee
then \eqref{1.0.7} is satisfied. Conversely, given $\frak{sl}(2,\R)$ matrices $X$ and $T$ satisfying \eqref{1.0.7}, we can obtain forms $\omega_1$, $\omega_2$ and $\omega_3$, satisfying \eqref{1.0.8} and each open and simply connected set $U\subseteq\R^2$ for which $\omega_1\wedge\omega_2\neq0$ everywhere is endowed with a PSS structure.

The CH equation has an extensive literature concerning analysis of PDEs and its connections with PSS as well. For example, \cite{const1998-1,const1998-2, const2000-1} and \cite{keti2015,reyes2002} explore these rather different aspects. However, the combination of these two natures has been barely explored. More precisely, the geometry of PSS determined by solutions of Cauchy problems involving certain equations of the type
\bb\label{1.0.11}
u_t-u_{txx}=\lambda uu_{xxx}+G(u,u_x,u_{xx})
\ee
seem to have begun very recently \cite{nilay,nazime,freire2023-1,freire2023-2}. In particular, \cite{freire2023-1,freire2023-2} investigate how the singularities of the solutions of the CH equation are transferred to the first fundamental form of a surface determined by a solution developing wave breaking.

The one-forms determined by the solutions of the CH equation known so far are not compatible with a second fundamental form having finite dependence of the derivatives of the solution \cite{tarcisio}. As a result, in practical terms this prevents us from constructing a second fundamental form of the PSS determined by the CH equation \cite{tarcisio}, and we cannot locally realise how surfaces determined by its solutions, including those emanating from Cauchy problems, can be locally immersed in $\mathbb{R}^3$ \cite{freire2023-2}.

The existence of a linear problem \eqref{1.0.5} for the CH equation anticipates its geometrical significance, implying the determination of the first fundamental form for an abstract surface. Unlike the CH equation, the existence of a $\frak{sl}(3,\R)$ ZCR suggests the impossibility for the solutions of the DP equation to determine surfaces in this way. Consequently, a geometric analysis study, such as those in \cite{nilay,freire2023-1,freire2023-2,nazime}, cannot be expected at first sight.

The situation, however, could not be more remarkably and dramatically unexpected. Not only {\it can} we construct first fundamental forms for PSS from the solutions of the DP equation, but unimaginably we also acquire a second fundamental form compatible with the first one. Therefore, it is possible to locally isometrically embed the surfaces defined by the solutions of the DP equation into the three-dimensional Euclidean space. This seems to be an unnoticed aspect reinforcing how different the CH and DP equations are.

As far as the author knows, the first work showing relations between the DP equation and PSS surfaces is \cite{keti2015}, where the following family of one-forms --$\mu$ is an arbitrary parameter-- were reported (see \cite[Example 2.8]{keti2015})
\bb\label{1.0.12}
\ba{lcl}
\omega_1&=&(u-u_{xx})dx+(u_x^2-2uu_x+uu_{xx})dt,\\
\\
\omega_2&=&\Big(\mu(u-u_{xx})\pm 2\sqrt{1+\mu^2}\Big)dx+\mu(u_x^2-2uu_x+uu_{xx})dt,\\
\\
\omega_3&=&\Big(\pm\sqrt{1+\mu^2}(u-u_{xx})+2\mu\Big)dx\pm\sqrt{1+\mu^2}(u_x^2-2uu_x+uu_{xx})dt.
\ea
\ee

Subsequently, in \cite{tarcisio} Castro Silva and Kamran studied the problem of immersions of PSS determined by the solutions of the class of equations \eqref{1.0.11}, whose classification had been previously carried out by Castro Silva and Tenenblat in \cite{keti2015}. One of the consequences of \cite[Theorem 1.1]{tarcisio} is the possibility of locally describe second fundamental forms compatible with \eqref{1.0.12}.

The purpose of this paper is to understand how a given initial datum influences the corresponding PSS determined by the solutions of the DP equation. Furthermore, we want to shed light on the geometry determined by solutions developing singularities in finite time.

The main issue of such an investigation is a sort of ``incompatibility'' between analysis and geometry: while the former deals with solutions with finite regularity, the later is mostly concerned with $C^\infty$ structures, that ultimately require $C^\infty$ solutions. More dramatically, the tools used to understand Cauchy problems involving \eqref{1.0.1} assume finite regularity of the solution, see \cite{const1998-1,const1998-2,const2000-1,escher2008, freirejde, freire-AML,henry,liu2006,liu2007,yin2003,yin2004,zhou}.

In addition to the aforesaid above, the machinery employed to tackle the problems from the analysis side requires us to see the DP as a non-local evolution equation, which only holds for certain function spaces. As a result, we have to make this approach compatible with geometry, which is the same to say that we have to interpret the non-local form of the DP as a an equation describing PSS.

In the next section we present our framework in order to address the two major problems mentioned, as well as our main results and how they are in light of the state of the art in both field (Analysis and Geometry of PDEs). Then we present a general picture of the structure of the manuscript.

\section{State of the art and main results}

In this section we present the main concepts and results, as well as we discuss the problems to be tackled in light of the current literature.

\subsection{The local and non-local forms of the DP equation and differentiability of solutions}\label{subsec2.1}

In the literature of PDEs and PSS, the focus is generally on considering smooth ($C^\infty$) solutions, leading to smooth ($C^\infty$) one-forms satisfying \eqref{1.0.8}, and thereby a smooth ($C^\infty$) metric given by \eqref{1.0.9}. Henceforth, by smooth we mean $C^\infty$. 

The solutions we wish to examine possess finite regularity. To make matters worse, the original form of the DP equation \eqref{1.0.3}, that sometimes will be referred to as ``local form'', is not the best suited for the analytical approaches employed herein. From an analytical standpoint, we should instead consider its non-local evolution form.
\bb\label{2.1.1}
u_t+uu_x+\f{3}{2}\p_x\Lambda^{-2}(u^2)=0.
\ee

As we shall see, the operator $\p_x$ commutes with both $\Lambda^2:=1-\p_x^2$ and its inverse $\Lambda^{-2}$. Applying $\Lambda^{2}$ to \eqref{2.1.1} we get \eqref{1.0.3}. The problem now is to give meaning to the operators $\Lambda^{-2}$ and $\p_x\Lambda^{-2}$. To this end, we shall make a digression and revisit Fourier analysis and some function spaces.

First of all, \eqref{2.1.1} can be seen as a dynamical system into a function space. Therefore, its solutions are nothing but a family of functions in some function space ${\cal B}$ parametrized by $t$, which is usually thought as time. 

Equation \eqref{2.1.1}, once solved for an interval of existence $I$, has to be then analysed in light of the function space the solutions belongs to. A crucial step in this endeavour is the determination of a suitable space for which the equation and its solutions have meaning. For \eqref{2.1.1} it is vital to find a function space compatible with the operator $\Lambda^{-2}$.

It is well known that the Fourier transform\footnote{The constant appearing in the definition of the Fourier transform may vary, which may lead to a different constant in the inversion transform. For example, compare \cite[page 200]{const-book} and \cite[page 222]{taylor}.} 
$$
{\cal F}(f)(k):=\f{1}{\sqrt{2\pi}}\int_\R e^{-ikx}f(x)dx,
$$
applies constant differential operators into polynomials and vice-versa, since $({\cal F}(\p_xf))(k)=ik\hat{f}(k)$, where ${\cal F}(f)(k)$ is replaced by $\hat{f}(k)$ for sake of simplicity. This is true provided that the domain of the involved operators is compatible with the Fourier transform.

Let us consider $f\in L^2(\R)$. Although such a function need not to be even $C^1$, we additionally assume that $f''\in L^2(\R)$. Integration by parts leads us to
$$
\hat{h}(k):={\cal F}(\Lambda^2(f))(k)=(1+k^2)\hat{f}(k),
$$
and then, we have
$$
\hat{f}(k)=\f{\hat{h}(k)}{1+k^2}=:\hat{G}(k)\hat{h}(k).
$$

Recalling that ${\cal F}(G\ast h)(k)=\sqrt{2\pi}\hat{G}(k)\hat{h}(k)$ (e.g, see \cite[Theorem 5.3]{const-book}) and $({\cal F}^{-1}(\hat{G}\hat{h}))(x)=(G\ast h)(x)/\sqrt{2\pi}$ whether it exits, then we conclude that
\bb\label{2.1.2}
G(x)=\f{e^{-|x|}}{2}.
\ee

This explains why very often \eqref{2.1.2} is referred as {\it Green function} of the (Helmholtz) operator $\Lambda^{2}=1-\p_x^2$. Hence, if $\delta$ denotes the Dirac delta distribution, then
$$
\int_\R G(x-y)(1-\p_y^2)f(y)dy=\int_\R\Big((1-\p_y^2)G(x-y)\Big)f(y)dy=\int_\R\delta(y-x)f(y)dy=f(x),
$$
in view of the self-adjointness of $\Lambda^{2}$ with respect to the $L^2$ inner product.

All the relations mentioned above can be schematically summarised in Figure 1.
\begin{figure}[h!]
\begin{center}
\begin{tikzcd}
\Lambda^2(\cdot)=(1-\p_x^2)(\cdot)\arrow{rrrrr}{\text{Fourier transform}}  &&&&&1+k^2\arrow{d}{\text{Inverse Fourier transform}}\\
\Lambda^{-2}(\cdot)=G\ast (\cdot)&&&&&\arrow{lllll}{\text{Inverse operator}} \ds{G(x)=\f{e^{-|x|}}{2}}
\end{tikzcd}
\end{center}
\caption{Actions of the operators $\Lambda^2$ and $\Lambda^{-2}$, and their relations with Fourier transform. The resulting action of the former operator on a function decreases its regularity, whereas its inverse, being given through a convolution process, improves the regularity of the resulting functions when compared with the original ones.}
\end{figure}

It is obvious that $(\p_x\Lambda^{2})(f)=\Lambda^2(\p_xf)$. In fact, the well known property of convolution 
$$
\p_x(f_1\ast f_2)=((\p_x f_1)\ast f_2)=(f_1\ast(\p_x f_2))
$$
implies that $\p_x$ commutes with $\Lambda^{-2}$.

Let us consider the class of smooth rapidly decaying functions ${\cal S}(\R)$ (Schwartz class, see \cite[page 125]{const-book}), with dual space ${\cal S}'(\R)$, whose members are called tempered distributions. We can define the Fourier transform $\hat{f}$ of a $f\in{\cal S}'(\R)$ throughout the relation
$\langle \hat{f},\phi\rangle_{L^2}=\langle f,\hat{\phi}\rangle_{L^2}$, for any $\phi\in {\cal S}(\R)$, where $\langle\cdot,\cdot\rangle_{L^2}$ denotes the usual $L^2(\R)$ inner product.
 
For $s\in\R$, we can define the Sobolev space of order $s$, namely,
$$
H^s(\R)=\{f\in{\cal S}'(\R),\,\,(1+k^2)^{s/2}\hat{f}(k)\in L^2(\R)\}.
$$

An extremely relevant result for our purposes is:
\begin{lemma}{\tt (Sobolev Embedding theorem, see \cite[Theorem 6.21]{const-book})}\label{lema2.1}
    If $s>k+1/2$, for some integer $k\geq0$, then all functions $f\in H^s(\R)$ belong to $C^k(\R)\cap L^\infty(\R)$.
\end{lemma}

The operator $\p_x$ is a bounded linear map from $H^s(\R)$ to $H^{s-1}(\R)$; $\Lambda^{t}:H^{s}(\R)\rightarrow H^{s-t}(\R)$ is a unitary isomorphism; ${\cal S}(\R)\subseteq H^s(\R)\subseteq {\cal S}'(\R)$ and ${\cal S}(\R)$ is dense in $H^s(\R)$, for any $s,t\in\R$. Finally, for $s\geq t$, we have the embedding $H^s(\R)\hookrightarrow H^t(\R)$.

Let us pay our promise and provide meaning to $\p_x\Lambda^{-2}$. It is straightforward to check that
$$
\ba{lcl}
\ds{\p_x\Big((\Lambda^{-2}(f))(x)\Big)}&=&\ds{\p_x\Big(\int_\R \f{e^{-|x-y|}}{2}f(y)dy\Big)=\int_\R \Big(-\f{\sign{(x-y)}e^{-|x-y|}}{2}\Big)f(y)dy}\\
\\
&=&\ds{\Big(\big(-\f{\sign{(\cdot)}e^{-|\cdot|}}{2}\big)\ast f\Big)(x)=\Big(\big(\p_x G\big)\ast f\Big)(x)}.
\ea
$$

The derivative above, under the sign of integral, has to be understood in the distributional sense and, as a result, we have the distribution
\bb\label{2.0.2}
G'(x)=-\f{\sign{x}}{2}e^{-|x|},
\ee
implying that the operator $\p_x\Lambda^{-2}:H^{s}(\R)\rightarrow H^{s+1}(\R)$ is given by $\p_x\Lambda^{-2}(f)(x)=(G'\ast f)(x)$.

In view of the precedent discussion, a strong solution of \eqref{2.1.1} is a function $u\in C^1(\R\times I)$, where $I\subseteq\R$ is the interval of existence of $u$, satisfying the equation identically. On the other, by a strong solution of \eqref{1.0.2} we mean a function satisfying \eqref{1.0.2} identically, but belonging to $C^{3,1}(\R\times I):=\{u:\R\times I\rightarrow\R; \p_x^iu,\,\p_x^j\p_tu\in C^0(\R\times I),\,\,0\leq i\leq 3,\,\,0\leq j\leq 2\}$. Therefore, these two forms of the DP equation are not equivalent, but they may agree on certain function spaces.

\subsection{Fundamental forms of a surface and immersions}\label{subsec2.2}

Let $\M$ be a domain in $\R^2$, that is, an open and simply connected set. For each $p\in\M$, $T_p\M$ denotes the tangent space to $\M$ at $p$. Assume that $e_1,e_2$ are two vector valued functions on $\M$, satisfying the following conditions at each point $p\in\M$:
\begin{itemize}
\item $\{e_1,e_2\}$ is orthonormal with respect to inner product $\langle\cdot,\cdot\rangle$ of $\R^3$;
\item $\Span\{e_1,e_2\}=T_p\M$;
\item the dual bases of $\{e_1,e_2\}$ is denoted by $\{\omega_1,\omega_2\}$ and, in passing, $\Span\{\omega_1,\omega_2\}=T_p^\ast\M$.
\end{itemize}

Both $\omega_1$ and $\omega_2$ are one-forms, that can be generically written as $\omega=f(x,t)dx+g(x,t)dt$, where $f$ and $g$ are certain functions. We say that $\omega$ is of class $C^k$ if and only if both $f$ and $g$ are $C^k$. They form a basis for the cotangent space $T_p^\ast\M$ as long as they are linearly independent (LI), which is measured by the condition $\omega_1\wedge\omega_2\big|_p\neq0$, where $\wedge$ denotes the wedge product (for further details, see \cite[page 39]{cle}).

The forms $\omega_1$ and $\omega_2$ uniquely define a third one $\omega_3$ \cite[Lemma 5.1, page 289]{neil}, named Levi-Civita connection form. These three one-forms satisfy the following structure equations:
\bb\label{2.2.1}
d\omega_1=\omega_2\wedge\omega_{21},\quad d\omega_2=\omega_1\wedge\omega_{12},
\ee
which determines $\omega_3:=\omega_{12}$, and the Gauss equation
\bb\label{2.2.2}
d\omega_3=-{\cal K}\,\omega_1\wedge\omega_2,
\ee
that gives the Gaussian curvature ${\cal K}$ of $\M$. For ${\cal K}=-1$, system \eqref{2.2.1}-\eqref{2.2.2} is identical to \eqref{1.0.8}. 

This triad of one-forms determines a surface from an abstract point of view, in the sense they give the curvature of a surface, and also determines its first fundamental form, vernacularly known as metric, given by
\bb\label{2.2.3}
I(v)=\omega_1(v)^2+\omega_2(v)^2,
\ee
with the convection $\al\be=\al\otimes\be$ and $\al^2=\al\al$, for any (one-)forms $\al$ and $\be$. For the definition of tensorial product $\otimes$, see \cite[page 39]{cle}. 

The intrinsic aspects mentioned above enables a being living on the surface calculate distance walked on $\M$, or angles. However, it does not tell us anything about how the surface looks like for three-, or higher, dimensional beings (like us) observing it.

The solution of the problem above is to look for connection forms, that inform the orientation of the surface into the three dimensional space, that is, it deals with an extrinsic aspect of the surface.

Since $\langle e_i,e_j\rangle$ is either $0$ or $1$, depending on whether $i=j$ or not, we then have
\bb\label{2.2.4}
\langle d e_i,e_j\rangle+\langle e_i,de_j\rangle=0,
\ee
where $d(\cdot)$ denotes the usual differential, and we can define new one-forms
\bb\label{2.2.5}
\omega_{ij}=\langle de_i,e_j\rangle.
\ee
From \eqref{2.2.5} we see that $\omega_{ij}=-\omega_{ji}$. In particular, $\omega_{ii}=0$.

The dual forms $\omega_1$ and $\omega_2$, jointly with the connection forms, satisfy the equations \eqref{2.2.1},
\bb\label{2.2.6}
\omega_1\wedge\omega_{13}+\omega_2\wedge\omega_{23}=0,
\ee
and
\bb\label{2.2.7}
d\omega_{12}=\omega_{13}\wedge\omega_{32},\quad d\omega_{13}=\omega_{12}\wedge\omega_{23},\quad d\omega_{23}=\omega_{21}\wedge\omega_{13}.
\ee

Since $\omega_1$ and $\omega_2$ are LI, from \eqref{2.2.6} and the Cartan lemma \cite[Lemma 2.49]{cle} we conclude the existence of functions $a,b,c$ such that
\bb\label{2.2.8}
\omega_{13}=a\omega_1+b\omega_2,\quad \omega_{23}=b\omega_1+c\omega_2.
\ee

\begin{lemma}(\cite[Theorem 10-19, page 232]{gug}, \cite[Theorem 10-18, page 232]{gug})\label{lemma2.2}
Let $\omega_1$, $\omega_2$, $\omega_{12}$, $\omega_{13}$, and $\omega_{23}$ be $C^1$ one-forms. Then they determine a local surface up to a euclidean motion if and only if $\omega_1\wedge\omega_2\neq0$ and equations \eqref{2.2.1}, \eqref{2.2.6} and \eqref{2.2.7} are satisfied.  
\end{lemma}

According to Lemma \ref{lemma2.2}, a set of $C^1$ one-forms in $\R^3$ satisfying \eqref{2.2.1}, \eqref{2.2.6} and \eqref{2.2.7} defines a surface $\M$ in the Euclidean space, at least locally. This result is known as the fundamental theorem of surface theory or Bonnet theorem, as referenced in \cite[theorem 11, page 143]{ilka}, \cite[theorem 4.39, page 127]{cle}, and \cite[theorem 4.24, page 153]{ku}. For us, the only case of truly interest is ${\cal K}=-1$, a situation for which the system \eqref{2.2.1}--\eqref{2.2.2} is equivalent to \eqref{1.0.8}.

Surfaces for which their Gaussian curvatures are constant and negative are called {\it pseudospherical} surfaces \cite[page 9]{cle}.

\subsection{Connections between PDEs and differential geometry of surfaces}\label{subsec2.3}

Chern and Tenenblat introduced the concept of pseudospherical equations \cite{chern}, establishing connections between certain PDEs and infinitely differentiable two-dimensional Riemannian manifolds. Their work was based on a remarkable observation by Sasaki \cite{sasaki}, who demonstrated links between solutions of special equations (obtained through the AKNS method \cite{akns}) and the intrinsic geometry of surfaces with constant Gaussian curvature ${\cal K}=-1$.

The two seminal papers \cite{chern,sasaki} served as seeds for a novel branch of surface geometry and relation to PDEs, profoundly influencing numerous subsequent research papers, as seen in \cite{beals,keti2015,tarcisio,cat,ding,nilay,freire-tito-sam,kah-book,kah-cag,kah,kamran,reyes2000,reyes2002,reyes2006-sel,reyes2006-jde,reyes2011,nazime} and references therein. For a more comprehensive treatment of the subject, see \cite[chapter 1]{keti-book}.

In \cite[Example 3.10]{keti2015} it was shown that smooth solutions of the local form of the DP give rise to metrics for a surface of constant Gaussian curvature ${\cal K}=-1$, see also \cite[Theorem 3.4]{keti2015}. The present work aims at studying, however, properties of the non-local form of the DP equation and how they connect with the geometry of surfaces. 

As mentioned in Subsection \ref{subsec2.1}, solutions of the two different forms of the DP are not necessarily coincident. Therefore, the first aspect we have to deal with is to construct a framework in which we can give geometric meaning to both \eqref{1.0.3} and \eqref{2.1.1} simultaneously. That being so, these two forms of the DP equation have to describe the same geometric objects (surfaces).

This lead us to the following recent definition.
\begin{definition}{\tt($C^k$ PSS modelled by ${\cal B}$ and ${\cal B}-$PSS equation, \cite[Definition 2.1]{freire2023-1})}\label{def2.1}
Let ${\cal B}$ be a function space. A differential equation 
\bb\label{2.3.1}
{\cal E}(x,t,u,u_{(1)},\cdots,u_{({n)}})=0,
\ee
for a dependent variable (function) $u\in{\cal B}$, is said to describe a pseudospherical surface of Gaussian curvature ${\cal K}=-1$ and class $C^k$ modeled by ${\cal B}$, $k\in\N$, or it is said to be of ${\cal B}-$pseudospherical type (${\cal B}$-PSS equation, for short), if it is a necessary and sufficient condition for the existence of functions $f_{ij}=f_{ij}(x,t,u,u_{(1)},\cdots,u_{(n)})$, $1\leq i\leq 3,\,\,,1\leq j\leq 2$, depending on $u$ and its derivatives up to a finite order $n$, such that:
\begin{itemize}
    \item ${\cal B}\subseteq C^k$;
    \item the functions $f_{ij}$ are $C^k$ with respect to their arguments;
    \item the forms 
    \bb\label{2.3.2}
\omega_i=f_{i1}dx+f_{i2}dt,\quad 1\leq i\leq 3,
\ee
satisfy the structure equations of a pseudospherical surface of Gaussian curvature ${\cal K}=-1$, that is,
\bb\label{2.3.3}
d\omega_1=\omega_3\wedge\omega_2,\quad d\omega_2=\omega_1\wedge\omega_3,\quad d\omega_3=\omega_1\wedge\omega_2;
\ee
    \item the condition $\omega_1\wedge\omega_2\neq0$ is satisfied.
\end{itemize}
\end{definition}

Whenever the function space is not exactly relevant, or clear, and no confusion is possible, we simply say PSS equation in lieu of ${\cal B}-$PSS equation.

\begin{remark}
Condition \eqref{2.3.3} is satisfied if and only if there exist functions $\mu_1$, $\mu_2$, and $\mu_3$ that vanish identically on the solutions of \eqref{2.3.1}. In other words, the conditions are as follows:
$$
d\omega_1-\omega_3\wedge\omega_2 = \mu_1dx\wedge dt, \,\,\,
d\omega_2-\omega_1\wedge\omega_3 = \mu_2dx\wedge dt, \,\,\,
d\omega_3-\omega_1\wedge\omega_2 = \mu_3dx\wedge dt,
$$
subject to the constraints $\mu_1\big|_{\eqref{2.3.1}} \equiv 0$, $\mu_2\big|_{\eqref{2.3.1}} \equiv 0$ and $\mu_3\big|_{\eqref{2.3.1}} \equiv 0$.

In simpler terms, the functions $\mu_1$, $\mu_2$, and $\mu_3$ must be zero when evaluated on the solutions of equation \eqref{2.3.1}.
\end{remark}

\begin{definition}{\tt(Generic solution, \cite[Definition 2.2]{freire2023-1})}\label{def2.2}
A function $u:U\rightarrow\R$ is called {\it generic solution} for the ${\cal B}-$PSS equation \eqref{2.3.1} if:
\begin{enumerate}\letra
\item $u\in{\cal B}$;
\item It is a solution of the equation;
\item The one-forms \eqref{2.3.2} are $C^k$ on $U$;
\item There exists at least a simply connected open set $\Omega\subseteq U$ such that $\omega_1\wedge\omega_2\big|_{p}\neq0$, for each $p\in\Omega$.
\end{enumerate}

Otherwise, $u$ is said to be {\it non-generic}.
\end{definition}

The notions introduced in these two definitions give us the basis for the our first result.

\begin{theorem}\label{teo2.1}
Both local and non-local forms of the DP equation (equations \eqref{1.0.4} and \eqref{2.1.1}, respectively) describe PSS of class $C^1$ modelled by ${\cal B}=C^{0}(H^{4}(\R),[0,T))\cap C^{1}(H^{3}(\R),[0,T))$, $T>0$, with one-forms given by \eqref{1.0.12}. In particular, in this class of solutions, the surfaces described by the local form of the equation are equivalent to that of the non-local form. Furthermore, a function in the class ${\cal B}$ is a solution of \eqref{1.0.4} if and only if it is a solution of \eqref{2.1.1}.
\end{theorem}

In view of the equivalence of solutions in the class ${\cal B}$, henceforth wherever we consider a solution for any of the forms of the DP equation, we simply refer to it as {\it solution}, without distinction between the two possible forms of the equation.

\begin{corollary}\label{cor2.1}
    Let $u$ be a solution of the DP equation. Then the first fundamental form defined by \eqref{1.0.12} is given by
    \bb\label{2.3.4}
    \ba{lcl}
    g&=&\ds{\Big[(1+\mu^2)(u-u_{xx})^2+4(1+\mu^2)\pm4\mu\sqrt{1+\mu^2}(u-u_{xx})\Big]}dx^2\\
    \\
    &+&\ds{2\Big[(1+\mu^2)(u-u_{xx})\pm 2\mu\sqrt{1+\mu^2}\Big](u_x^2-2uu_x+uu_{xx})dxdt}\\
    \\
    &+&\ds{(1+\mu^2)(u_x^2-2uu_x+uu_{xx})^2dt^2.}
    \ea
    \ee
\end{corollary}

While theorem \ref{teo2.1} concerns to a structural aspect of the equation (in the sense both forms may describe the same geometric object), our next result is of more qualitative nature and perhaps more remarkable, since it states that any non-trivial initial datum in $H^4(\R)$ necessarily endows some domain of $\R^2$ with the structure of a PSS. By non-trivial we mean a function that does not vanish identically.

\begin{theorem}{\tt (Existence of PSS on a strip of finite height)}\label{teo2.2} Let $u_0\in H^4(\R)$ be a non-trivial initial datum. Then there exists a strip ${\cal S}$, uniquely determined by $u_0$, such that:
\begin{enumerate}\letra
    \item The one-forms \eqref{1.0.12} are defined on ${\cal S}$, where $u$ is the unique solution of the Cauchy problem
$$
\left\{\ba{l}
\ds{u_t+uu_u=-\p_x(1-\p_x^2)^{-2}\Big(\f{3}{2}u^2} \Big),\\
\\
u(x,0)=u_0(x).
\ea\right.
$$
    \item There exists at least one simply connected set ${\cal C}\subseteq{\cal S}$ such that $\omega_1\wedge\omega_2\big|_{p}\neq0$, for every $p\in{\cal C}$. As a result, ${\cal C}$ has the structure of a $C^1-$PSS modelled by ${\cal B}\subseteq C^{3,1}(\R\times[0,T))$, where $T>0$ is determined by $u_0$.
\end{enumerate}
\end{theorem}

The second part of the last theorem proves the following result.

\begin{corollary}\label{cor2.2}
Any solution of the DP equation subject to an initial datum $u_0\in H^4(\R)$ is a generic solution in the sense of Definition \ref{def2.2}.
\end{corollary}

In the proof of Theorem \ref{teo2.2} we will see that ${\cal S}=\R\times(0,T)$, where $T>0$ is uniquely determined by $u_0$. While the coordinate $t$ may be restrict to some range of values, we do not have any sort of restriction with respect to $x$. Our next result takes a look on asymptotic behaviour of the metric determined by \eqref{1.0.12} for $|x|\rightarrow\infty$. Below $\pi_1$ denotes the canonical projection on the first variable, and
$$
\supp(f)=\overline{\{x\in\R;\,\,f(x)\neq0\}}
$$
denotes the support of the function $f$.

\begin{theorem}{\tt (Asymptotic behaviour of metrics)}\label{teo2.3}
Suppose that $u_0\in H^4(\R)$ is a non-trivial, compactly supported initial datum, with $[a,b]:=\supp(u_0)$, and ${\cal S}$ be the strip determined by theorem \ref{teo2.2}. Then there exist two curves $\gamma_+,\gamma_-:[0,T)\rightarrow\overline{{\cal S}}$, for some $T>0$ determined by $u_0$, such that:
\begin{enumerate}\letra
    \item $\pi_1(\gamma_-(0))=a$ and $\pi_1(\gamma_+(0))=b$;
    \item $\gamma_\pm'(t)\neq 0$, $t>0$;
    \item $\pi_1(\gamma_-(t))<\pi_1(\gamma_+(t))$, for any $t\in[0,T)$;
    \item On the right of $\gamma_+$, the first fundamental form is given by
    \bb\label{2.3.5}
    g=4(1+\mu^2)dx^2\pm4\mu\sqrt{1+\mu^2}E(t)^2 e^{-2x}dxdt+(1+\mu^2)E(t)^4 e^{-4x}dt,
    \ee
    where
    \bb\label{2.3.6}
    E(t)=\int_{\pi_1(\gamma_-(t))}^{\pi_1(\gamma_+(t))}e^{- x} m(x,t)dx;
    \ee
    \item On the left of $\gamma_-$ we no longer have a metric defined on ${\cal S}$, though the forms $\omega_1$ and $\omega_2$ are still defined everywhere.
\end{enumerate}    
\end{theorem}

Denoting by $(g)$ the matrix of the first fundamental form, from \eqref{2.3.5} we have
$$
(g)=\begin{pmatrix}
4(1+\mu^2)& \pm2\mu\sqrt{1+\mu^2}E(t)^2 e^{-2x}\\
\\
\pm2\mu\sqrt{1+\mu^2}E(t)^2 e^{-2x}& (1+\mu^2)E(t)^4 e^{-4x}
\end{pmatrix}=
\begin{pmatrix}
4(1+\mu^2)& 0\\
\\
0& 0
\end{pmatrix}+O(e^{-2x}),
$$
that is, the matrix of the metric $g$ is an $O(e^{-2x})$ perturbation of the singular matrix $\Diag(4(1+\mu^2),0)$ as $x\rightarrow+\infty$, showing that the metric becomes asymptotically singular on the right side of the strip ${\cal S}$, for each fixed $t\in(0,T)$.

Not so long ago, the literature on PSS equations was primarily focused on abstract aspects of surfaces determined by these equations, paying little to no attention to how these surfaces are immersed in the three-dimensional Euclidean space. This becomes an intriguing question because it has long been known that complete realisation of such surfaces in the usual Euclidean space is impossible due to Hilbert's theorem, which forbids the isometric immersion of a complete surface with negative curvature into $\R^3$ \cite[page 439]{neil}, \cite{milnor}, \cite[section 5-11]{carmo}. Consequently, the only viable way to realise these surfaces is through local isometric immersions.

The problem of immersing PSS surfaces determined by solutions of PSS equations emerged fairly recently in a series of papers by Kahouadji, Kamran, and Tenenblat \cite{kah-book, kah-cag, kah}. Later on, Castro Silva and Kamran \cite{tarcisio} addressed the immersion problem of PSS determined by a class of equations studied in \cite{keti2015}.

The findings reported in \cite{kah-book, kah-cag, kah, tarcisio} strongly suggest that the second fundamental forms of PSS equations are most likely either independent of the solutions of the equations (universal) or dependent on an infinite number of derivatives of the dependent variable. There are very few examples of equations that lie between these two extremes, where the second fundamental form depends only on a finite number of the dependent variable and its derivatives. In particular, so far there is no equation of the type \eqref{1.0.11} for which a known second fundamental form depends on $u$ or its derivatives up some finite order.

The results reported in \cite{tarcisio} show that the situation for the equations studied in \cite{keti2015} is not better. In fact, none of the families classified in \cite{keti2015} have a second fundamental form depending on a finite jet. However, the Degasperis-Procesi (DP) equation is a PSS equation \cite[Example 2.8]{keti2015} and has a universal second fundamental form \cite[Theorem 1.1]{tarcisio}, that is, its coefficients do not depend neither on $u$, nor on its derivatives of any other.

Our next result establishes a link between immersions and Cauchy problems.

\begin{theorem}{\tt (Existence of a second fundamental form)}\label{teo2.4}
Let $u_0\in H^4(\R)$ be a non-trivial initial datum, ${\cal S}$ be the strip determined by theorem \ref{teo2.2} and $U$ be a simply connected component of ${\cal S}$ for which the forms $\omega_1$ and $\omega_2$ \eqref{1.0.12} are linearly independent. Then we can find smooth real valued functions $a,b,c$, locally defined on $U$, such that
\eqref{2.2.8} are connection forms that, jointly with $\omega_1$ and $\omega_2$, determine the first two fundamental forms of a PSS.
\end{theorem}

It is well known that solutions of the DP equation may blow up when $t$ approaches a certain value $T_0$, see \cite[Theorem 4.2]{liu2006}. Our next result tackle the problem of singularities of the metric. By a singularity of the metric we mean a situation in which either the first fundamental form is well defined, but no longer a definite positive bi-linear form, or it blows up. The first situation, in fact, is treated in the proof of Theorem \ref{teo2.2} and in Theorem \ref{teo2.3}. The second case is addressed right now.

\begin{theorem}{\tt (Blow up of metrics)}\label{teo2.5}
Let $u_0\in H^4(\R)$, $m_0=u_0-u_0''$, and ${\cal S}$ the corresponding strip determined by Theorem \ref{teo2.2}. In addition, assume the existence of a point $x_0$ such that:
\begin{enumerate}\letra
    \item $m_0(x)\leq 0$, for $x\ge x_0$;
    \item $m_0(x)\geq0$, for $x\leq x_0$.
\end{enumerate}

Under the conditions above, the metric determined by \eqref{1.0.12} blows up in finite height, in the following sense: there exists a curve $\gamma:[0,T_0)\rightarrow\R^2$, such that $\gamma(0)=x_0$, $T_0<\infty$, $\gamma(0,T_0)\subseteq{\cal S}$ and
$$
\lim_{t\nnearrow T_0}g_{22}(\gamma(t))=+\infty,
$$
for any value of $\mu$. In case $\mu\neq0$, we also have $g_{12}(\gamma(t))\rightarrow+\infty$ as $t\nnearrow T_0$.
\end{theorem}

All the results considered until now have a local nature, in the sense that the strip determined in Theorem \ref{teo2.2} has a finite height. A natural question is: can we have the forms \eqref{1.0.12} defined over ${\cal S}=\R\times(0,\infty)$?
While theorem \ref{teo2.5} provides a negative answer under certain circumstances, a change of conditions drives us to a completely different scenario.

\begin{theorem}{\tt (Global domain for the first fundamental form)}\label{teo2.6} Assume that $m_0\in H^2(\R)\cap L^1(\R)$ is non-trivial and either non-negative or non-positive, and $u$ be the corresponding solution of the DP equation subject to $u(x,0)=u_0(x)$, where $u_0=G\ast m_0$, where $G$ is given by \eqref{2.1.2}. Then the one forms \eqref{1.0.12} are defined on ${\cal S}=\R\times(0,\infty)$.
\end{theorem}

\subsection{Outline of the manuscript}

The paper is structured as follows: in the next section we revisit some crucial results regarding qualitative nature of solutions of the DP equation, as well as we prove the existence of a bijection playing vital role in the proofs of theorems \ref{teo2.3}, \ref{teo2.5} and \ref{teo2.6}. 

In section \ref{sec4} we prove theorem \ref{teo2.1} and its corollary, showing that both formulations of the DP equation are PSS equations. This is the cornerstone result for establishing the uniqueness of domains possessing the structure of a PSS surface determined by the unique solution of the DP subject to an initial datum (theorem \ref{teo2.2}) and pay more attention to compactly supported initial condition (theorem \ref{teo2.3}), that are proved in section \ref{sec5}.

In section \ref{sec6} we prove that the abstract surface determined by Theorem \ref{teo2.2} can be locally immersed in the Euclidean space by determining a (locally determined) second fundamental form (Theorem \ref{teo2.4}). 

In section \ref{sec7} we consider what happens to a surface determined by an initial datum giving rise to a solution breaking in finite time (Theorem \ref{teo2.5}) and then we consider a result of more global nature, expressed in Theorem \ref{teo2.6}.

Finally, due to the triad \eqref{1.0.12} it is a natural temptation to wonder whether \eqref{1.0.4} might be geometrically integrable. This is a particularly intriguing question because the DP equation has the Lax pair \eqref{1.0.6}, which would prevent it to be geometrically integrable. On the other hand, the fact that \eqref{1.0.4} depends on a parameter is a strong indication of geometric integrability. We shed light on this question in section \ref{sec8}.

Our results are discussed in section \ref{sec9} and our conclusions are given in section \ref{sec10}.

\section{Qualitative results}\label{sec2}

The DP equation has an infinite hierarchy of conserved quantities. This is a foregone consequence of the existence of a bi-Hamiltonian structure \cite{dehoho}. Some of them are the following:
\bb\label{3.0.1}
E_1(t)=\int_\R m(x,t)dx,\quad E_2(t)=\int_\R m(x,t)v(x,t)dx,\quad E_3(t)=\int_\R u(x,t)^3dx,
\ee
where $v(x,t):=((4-\p_x^2)^{-1}u)(x,t)$.

Moreover, using the momentum, or potential, $m(x,t)=u(x,t)-u_{xx}(x,t)$, we can rewrite the DP equation \eqref{1.0.3} as
\bb\label{3.0.2}
m_t+um_x+3u_xm=0.
\ee

\begin{lemma}{\tt(\cite[Theorem 2.2]{yin2003})}\label{lema3.1} Given $u_0\in H^s(\R)$, $s>3/2$, there exists a maximal value $T=T(u_0)>0$ and a unique solution $u\in C^0(H^s(\R),[0,T))\cap C^1(H^{s-1}(\R),[0,T))$ to the problem
\bb\label{3.0.3}
\left\{\ba{l}
\ds{u_t+uu_u=-\p_x(1-\p_x^2)^{-2}\Big(\f{3}{2}u^2} \Big),\\
\\
u(x,0)=u_0(x).
\ea\right.
\ee

Moreover, the solution depends continually on the initial datum and the maximal time of existence $T$ can be chosen independently of $s$. 
\end{lemma}

A consequence of the precedent result is that any solution emanating from an initial datum is defined on a certain strip ${\cal S}=\R\times(0,T)$ for some $T>0$ only determined by the initial datum, but not $s$. This is a local result in nature.



\begin{theorem}\label{teo3.1}
    Let $u_0\in H^4(\R)$, $u\in C^0(H^4(\R),[0,T))\cap C^1(H^{3}(\R),[0,T))$ be the corresponding solution of \eqref{2.0.2} and $m(x,t)=u(x,t)-u _{xx}(x,t)$. Then the flux of $u$ defines a bijection $\varphi:\R\times[0,T)\rightarrow \R\times[0,T)$ such that:
    \begin{enumerate}\letra
        \item $\R\times\{0\}$ is invariant under $\varphi$;
        \item $\varphi\big|_{\R\times(0,T)}:\R\times(0,T)\rightarrow \R\times(0,T)$ is a $C^1$ diffeomorphism;
        \item $(m\circ\varphi)(x,t)q_x(x,t)^3=m_0(x)$, $(x,t)\in\R\times[0,T)$, where
        \bb\label{3.0.4}
        q_x(x,t)=e^{\ds{\int_0^t u_x(\varphi(x,s))ds}}.
        \ee
    \end{enumerate}
\end{theorem}

\begin{proof}
    Consider the auxiliary Cauchy problem:
\begin{equation}\label{3.0.5}
\begin{cases}
q_t(x,t) = u(q,t),\\
\\
q(x,0) = x.
\end{cases}
\end{equation}

By \cite[Theorem 3.1]{const2000-1}, see also \cite[Lemma 3.2]{yin2004}, this problem has a unique solution $q\in C^1(\R\times[0,T),\R)$. Moreover, for each fixed $t$, the function $q(\cdot,t):\R\rightarrow\R$ defines a one-parameter family of increasing diffeomorphisms, since $q_x(x,t)$, given by \eqref{3.0.4}, is strictly positive.

Let us define the mapping $\varphi:\R\times[0,T)\rightarrow\R\times[0,T)$ by $\varphi(x,t) = (q(x,t),t)$. It can be easily verified that $\varphi$ is a continuous bijection, and $\varphi(x,0) = (q(x,0),0) = (x,0)$. Denote by $J_\varphi(x,t)$ the Jacobian matrix of $\varphi$ at the point $(x,t)$. Through a simple calculation, it can be shown that $\det{J_\varphi(x,t)} = q_x(x,t)$. Combining this result with \eqref{3.0.4}, we deduce that $\varphi$ is a local diffeomorphism when restricted to the set $V = \R\times(0,T)$. Since $\varphi(V) = V$, $\varphi\big|_V$ is $C^1$ everywhere, and the function $q(\cdot,t)$ is a $C^1$ diffeomorphism, we can conclude that $\varphi\big|_V$ is a global $C^1$ diffeomorphism.

It remains to be proved that $m(\varphi(x,t))q_x(x,t)^3=m_0(x)$. After fixing $x\in\R$, we can derive the following expression:
\[
\frac{d}{dt}m(\varphi(x,t)) = (m_t + m_x q_t)(\varphi(x,t)) = 2u_x(\varphi(x,t)) m(\varphi(x,t)),
\]
where we used equations \eqref{3.0.2} and \eqref{3.0.5}. The result is then obtained by integrating the above relation from $0$ to $t$, and taking \eqref{3.0.4} into account.
\end{proof}

\section{The DP equations as a pseudospherical equation: proof of Theorem \ref{teo2.1} and its corollary}\label{sec4}

First of all, let $T>0$ and consider the function space ${\cal B}:=C^{0}(H^{4}(\R),[0,T))\cap C^{1}(H^{3}(\R),[0,T))$. By the Sobolev Embedding Theorem (see Lemma \ref{lema2.1}), for a fixed $t\in(0,T)$, we have $u(\cdot,t)\in H^{4}$ and $u_t(\cdot,t)\in H^{3}$ and, in particular, ${\cal B}\subseteq C^{3,1}(\R\times[0,T))\subseteq C^1(\R\times[0,T))$. 

{\bf Proof of Theorem \ref{teo2.1}.} Let $v\in{\cal B}$. A straightforward calculation shows that
\bb\label{4.0.1}
(1-\p_x^2)\Big(v_t+vv_x+\f{3}{2}\p_x\Lambda^{-2}\Big(v^2\Big)\Big)=v_t-v_{txx}+4vv_x-3v_xv_{xx}-vv_{xxx}.
\ee

From \eqref{4.0.1} we see that $u\in {\cal B}$ is a solution of the local form of the DP equation \eqref{1.0.4} if and only if it is a solution of the non-local form \eqref{2.1.1}.

Let $u\in{\cal B}$ and consider the one-forms \eqref{1.0.12}. After some reckoning, we get
\bb\label{4.0.2}
\ba{lcl}
d\omega_1-\omega_3\wedge\omega_2&=&\ds{(1-\p_x^2)\Big(u_t+uu_x+\f{3}{2}\p_x\Lambda^{-2}\Big(u^2\Big)\Big)dx\wedge dt,}\\
\\
d\omega_2-\omega_1\wedge\omega_3&=&0,\\
\\
d\omega_3-\omega_1\wedge\omega_2&=&\ds{-(1-\p_x^2)\Big(u_t+uu_x+\f{3}{2}\p_x\Lambda^{-2}\Big(u^2\Big)\Big)dx\wedge dt},
\ea
\ee
that, in view of \eqref{4.0.1}, is equivalent to
\bb\label{4.0.3}
\ba{lcl}
d\omega_1-\omega_3\wedge\omega_2&=&\ds{\Big(u_t-u_{txx}+4uu_x-3u_xu_{xx}-uu_{xxx}\Big)dx\wedge dt,}\\
\\
d\omega_2-\omega_1\wedge\omega_3&=&0,\\
\\
d\omega_3-\omega_1\wedge\omega_2&=&\ds{-\Big(u_t-u_{txx}+4uu_x-3u_xu_{xx}-uu_{xxx}\Big)dx\wedge dt}.
\ea
\ee

As a result, if $u$ is a solution of some of any of the two possible forms of the DP, then it is a solution of another and both \eqref{4.0.2} and \eqref{4.0.3} reduce to \eqref{1.0.8}.

Finally, we observe that
\bb\label{4.0.4}
\omega_1\wedge\omega_2=\mp 2\sqrt{1+\mu^2}(u_x^2-2uu_x+uu_{xx}),
\ee
meaning that $\omega_1\wedge\omega_2$ is not identically zero.\hfill$\square$

{\bf Proof of Corollary \ref{cor2.1}.} It is immediate and follows from \eqref{1.0.12} and \eqref{1.0.9}. \hfill$\square$

\section{Surfaces determined by Cauchy problems: proof of theorems \ref{teo2.2}--\ref{teo2.3}}\label{sec5}

Here we establish bridges connecting qualitative aspects of solutions with the geometry determined by them.

{\bf Proof of Theorem \ref{teo2.2}.} Applying Lemma \ref{lema3.1} with $s=4$ we conclude the existence of solution $u\in C^{0}(H^{4}(\R),[0,T))\cap C^{1}(H^{3}(\R),[0,T))$ for the DP equation, where $T$ is uniquely determined by $u_0$. As a result, the function $u$, for $t\neq0$, is defined on the open set ${\cal S}=\R\times(0,T)$, which is completely determined by $T$ and, ultimately, by $u_0$. Therefore, the forms \eqref{1.0.12} are $C^1$, satisfy \eqref{1.0.8}, and are defined everywhere on ${\cal S}$. Moreover, since $u_0$ is non-trivial, then $u$ cannot vanish everywhere on ${\cal S}$.

Let us show the existence of a connected set $\emptyset\neq{\cal C}\subseteq{\cal S}$ endowed with the structure of a PSS.

In view of \eqref{4.0.4} and the regularity of the solution $u$, such a set would not exist if and only if $\omega_1\wedge\omega_2=0$ on ${\cal S}$, which is the same to say that $u$ is a solution of the DP equation satisfying the constraint
\bb\label{5.0.1}
uu_{xx}-2uu_x+u_x^2=0.
\ee

The latter equation can be integrated once, and we then obtain
\bb\label{5.0.2}
uu_x-u^2=f(t),
\ee
for some continuous function $f$. Since both $u$ and $u_x$ vanishes at infinity and $f$ depends only on $t$, we conclude that $f(t)\equiv0$. Bringing this into \eqref{5.0.2} we can easily integrate the result and obtain
$$
u(x,t)=h(t)e^x,
$$
for some continuous function $h$. This would then imply that $u_0(x)=h(0)e^x$, which contradicts the fact that $u_0\in L^2(\R)$ and is non-trivial.

The contradiction pointed out above implies $(uu_{xx}-2uu_x+u_x^2)(x_0,t_0)\neq0$, for some $(x_0,t_0)\in{\cal S}$. Let $F(x,t)=(uu_{xx}-2uu_x+u_x^2)(x,t)$. Without loss of generality, we may assume $F(x_0,t_0)>0$. Since $u\in C^{0}(H^{4}(\R),[0,T))\cap C^{1}(H^{3}(\R),[0,T))$, then $F\in C^1(\R\times[0,T))$, wherefrom we conclude the existence of $\epsilon>0$ such that $B_\epsilon(x_0,t_0)\subseteq{\cal S}$, where $B_\epsilon(x_0,t_0)$ is the usual disc of centre $(x_0,t_0)$ and radius $\epsilon$, and
$$
F\big|_{B_\epsilon(x_0,t_0)}>0.
$$

To conclude, let 
$$
U:=\{(x,t)\in{\cal S};\,\,\,F(x,t)=0\}.
$$
The set $U$ is closed by construction, and clearly we have $B_\epsilon(x_0,t_0)\subseteq {\cal S}\setminus U$. Let ${\cal C}$ be the connected component of ${\cal S}\setminus U$ containing the disc $B_\epsilon(x_0,t_0)$. By construction, for any $p\in{\cal C}$, we have $\omega_1\wedge\omega_2\big|_p\neq0$. \hfill$\square$

{\bf Proof of Theorem \ref{teo2.3}.} Let $u$ be the corresponding solution of the DP equation subject to $u(x,0)=u_0(x)$, and $\varphi$ be the bijection given in theorem \ref{teo2.1}. Define $\gamma_-,\,\gamma_+:[0,T)\rightarrow\R^2$ by $\gamma_-(t)=\varphi(a,t)$ and $\gamma_+(t)=\varphi(b,t)$. By construction, we have $\gamma_+(t)=(q(b,t),t)$ and $\gamma_-(t)=(q(a,t),t)$, and thus, $\pi_1(\gamma_+(t))=q(b,t)$ and $\pi_1(\gamma_-(t))=q(a,t)$. Then these two curves clearly satisfy the conditions in $a)$, $b)$ and $c)$ and their images lie on $\overline{{\cal S}}$. From them we can obtain the desired curves for Theorem \ref{teo2.3}.

From \cite[Theorem 2.5]{henry}, the conditions on $u_0$ imply that $u$ can be expressed as
$$
u(x,t)=\left\{
\ba{lcl}
\ds{\f{E_+(t)}{2}e^{-x}},\quad\text{for}\quad x>q(b,t),\\
\\
\ds{\f{E_-(t)}{2}e^{+x}},\quad\text{for}\quad x<q(a,t),
\ea
\right.
$$
where
$$
E_\pm(t)=\int_{q(a,t)}^{q(b,t)}e^{\mp x}m(y,t)dy=\int_{\pi_1(\gamma_-(t))}^{\pi_1(\gamma_+(t))}e^{\mp x}m(y,t)dy.
$$

For $x>q(b,t)$, we have $uu_{xx}-2uu_x+u_x^2=4u^2=E_+(t)^2e^{-2x}=E(t)^2e^{-2x}$, where $E$ is given by \eqref{2.3.6}, and $m=0$, that substituted into \eqref{2.3.4}, implies \eqref{2.3.5}. For $x<q(a,t)$, we again have $m=0$, but now $uu_{xx}-2uu_x+u_x^2=0$, that yields $g_{12}=g_{22}=0$.
\hfill$\square$

\section{Proof of theorem \ref{teo2.4}: Immersions of surfaces determined by an initial datum}\label{sec6}

{\bf Proof of Theorem \ref{teo2.4}.} Theorem \ref{teo2.2} ensures the existence of $C^1$ one-forms defined on a strip ${\cal S}$ endowing some simply connected set ${\cal C}$ with the structure of a PSS. Let us then take $U:={\cal C}$. 

Our equation belongs to the class of equations described in \cite[Theorem 3.4]{keti2015} with parameters $\gamma=2$, $\lambda=1$, $\beta=0$, $m_1=2$ and $m_2=0$. If we assume $\mu=0$, then \cite[Proposition 3.7]{tarcisio} tells us that the second fundamental form is given by \eqref{2.2.8} with functions $a(x)=\pm\sqrt{L(2m)}$, $b(x)=-b_0e^{4x}$ and $c(x)=a(x)-a'(x)$, where $L(x)=\sigma e^{2z}-b_0 e^{4z}-1$, $\sigma,b_0\in\R$ are constants satisfying $\sigma^2>4b_0^2$ and $\sigma>0$.

Moreover, the connection forms are defined on $(x,t)\in U$ provided that
    $$
    \ln\sqrt{\f{\sigma-\sqrt{\sigma^2-4b_0^2}}{2b_0^2}}<2x<\ln\sqrt{\f{\sigma+\sqrt{\sigma^2-4b_0^2}}{2b_0^2}}.
    $$
Varying the parameters $\sigma$ and $b_0$, we can define connection forms
\bb\label{6.0.1}
\omega_{13}=a\omega_1+b\omega_2,\quad \omega_{23}=b\omega_1+c\omega_2
\ee
on each point $(x,t)\in U$.

    For $\mu\neq0$, the functions $a,b,c$ are smooth functions of $z=2x$, given by \cite[Proposition 3.7]{tarcisio} 
    $$
    \begin{array}{lcl}
    a&=&\ds{\frac{1}{2\mu}\left[\pm\mu\sqrt{\varDelta}-\left(\mu^{2}-1\right)b+b_0 e^{2\left(m_{1}x\right)}\right],}\\
    \\
    c&=&\ds{\frac{1}{2\mu}\left[\pm\mu\sqrt{\varDelta}+\left(\mu^{2}-1\right)b-b_0 e^{2\left(m_{1}x\right)}\right]},\\
    \\
    \varDelta&=&{\displaystyle \frac{\left[\left(\mu^{2}-1\right)b-b_0 e^{2\left(m_{1}x\right)}\right]^{2}-4\mu^{2}\left(1-b^{2}\right)}{\mu^{2}}}>0,
    \end{array}
    $$
where $b$ satisfies the ordinary differential equation 
$$
\begin{array}{l}
\ds{\left[\mu\left(1+\mu^{2}\right)\sqrt{\varDelta}\pm\left(\mu^{2}+1\right)^{2}b\mp\left(\mu^{2}-1\right)b_0 e^{2\left(m_{1}x\right)}\right]b^{'}}\\
\\
\ds{+2\left[-\mu\left(1+\mu^{2}\right)\sqrt{\varDelta}\mp b_0\left(\mu^{2}-1\right)e^{2\left(m_{1}x\right)}\right]b\pm2b_0^{2}e^{4\left(m_{1}x\right)}=0.}
\end{array}
$$

Arguing similarly as in \cite[pages 36--37]{tarcisio}, we can show that the connection forms can be locally defined on each point of $U$.
\hfill$\square$

\section{Breakdown of surfaces and global aspects: proof of theorems \ref{teo2.5} and \ref{teo2.6}}\label{sec7}

In this section we prove two results of opposite nature: the first one concerns with the collapse of the surface, by showing that its metric can only be well-defined on a strip of finite height and blows up near some height. The second result tells us that the dual coframe can exist on the upper half plane $\R\times(0,\infty)$ provide that the initial momentum is a one sign $L^1$ function.

{\bf Proof of Theorem \ref{teo2.5}.} Let $T$ be the lifespan of the corresponding solution $u$ of the DP equation subject to $u(x,0)=u_0(x)$, ${\cal S}$ be the strip given in Theorem \ref{teo2.2} and $\varphi$ be the bijection given in Theorem \ref{teo3.1}. 

Since $m=u-u_{xx}$, then $u=G\ast m$, where $G$ is given by \eqref{2.1.2} and $u_x=G'\ast m$. Then, we have the representation formulae
$$
u(x,t)=\f{e^{-x}}{2}\int_{-\infty}^x e^z m(z,t)dz+\f{e^{x}}{2}\int^{\infty}_x e^{-z} m(z,t)dz
$$
and
$$
u_x(x,t)=-\f{e^{-x}}{2}\int_{-\infty}^x e^z m(z,t)dz+\f{e^{x}}{2}\int^{\infty}_x e^{-z} m(z,t)dz.
$$

In particular, we have
$$
(u+u_x)(x,t)=\f{e^{-x}}{2}\int_{-\infty}^x e^z m(z,t)dz.
$$

Define $\gamma(t):=\varphi(x_0,t)$ and $q(t):=q(x_0,t)=\pi_1(\varphi(x_0,t))$, where $\pi_1$ is the projection on the first component, and
\bb\label{7.0.1}
I(t):=(u+u_x)(\gamma(t))=e^{q(t)}\int_{q(t)}^\infty e^{-z}m(z,t)dz.
\ee

Under the conditions on $m_0(\cdot)$ and by Theorem \ref{teo3.1}, we have
$$m(\gamma(t))=m(\varphi(x_0,t))=m_0(x_0)q_x^{-3}(x_0,t)=0$$
and
\bb\label{7.0.2}
I_0:=I(0)=e^{x_0}\int_{x_0}^\infty e^{-z}m_0(z)dx<0.
\ee

Besides, from Appendix \ref{ap1}
\bb\label{7.0.3}
\f{d}{dt}I(t)\leq u(q(t),t)^2-u_x(q(t),t)^2-\f{e^{q(t)}}{2}\int_{q(t)}^\infty e^{-z}(u(z,t)^2-u_x(z,t)^2)dz<0.
\ee

From \eqref{7.0.1} and \eqref{7.0.3} we infer that $I(t)$ is strictly decreasing and negative for any $t$ for which $\gamma(t)$ is defined.

Consider the following new functions
\bb\label{7.0.4}
f(t):=(f_{12}\circ\gamma)(t)=(u(q(t),t)-u_x(q(t),t))^2
\ee
and
\bb\label{7.0.5}
g(t):=u(q(t),t)-u_x(q(t),t).
\ee

Differentiating $g$ with respect to $t$ and taking \eqref{3.0.5} into account, we get (we omit the point $(q(t),t)$ for simplicity)
$$
g'(t)=(u_t-u_{tx})+q'(t)(u_{x}-u_{xx})=u_x^2-\f{3}{2}u^2+\f{3}{2}(\Lambda^{-2}-\p_x\Lambda^{-2})u^2,
$$
where we used \eqref{2.1.1} and the fact that 
$$
u_{tx}=\f{3}{2}u^2-u_x^2-uu_{xx}-\f{3}{2}\Lambda^{-2}u^2.
$$

On the other hand, we have
$$
((\Lambda^{-2}-\p_x\Lambda^{-2})u^2)(q(t),t)=e^{-q(t)}\int_{-\infty}^{q(t)}e^z u(z,t)dz,
$$
that, substituted into the previous equation, gives
\bb\label{7.0.6}
\ba{l}
g'(t)=\ds{u_x(q(t),t)^2-\f{3}{2}u(q(t),t)^2+\f{3}{2}e^{-q(t)}\int_{-\infty}^{q(t)}e^z u(z,t)dz=u_x(q(t),t)^2-\f{3}{2}u(q(t),t)^2}\\
\\
\ds{+\f{e^{-q(t)}}{2}\int_{-\infty}^{q(t)}e^z(u(z,t)^2-u_x(z,t)^2)dz+e^{-q(t)}\int_{-\infty}^{q(t)}e^z\Big(u(z,t)^2+\f{u_x(z,t)^2}{2}\Big)dz}.
\ea
\ee

From \cite[page 347]{const2000-1} we have
$$
\f{u(q(t),t)^2}{2}\leq e^{-q(t)}\int_{-\infty}^{q(t)}e^z\Big(u(z,t)^2+\f{u_x(z,t)^2}{2}\Big)dz,
$$
that substituted into \eqref{7.0.6}, implies
$$
g'(t)\geq u_x(\gamma(t))^2-u(\gamma(t))^2+\f{e^{-q(t)}}{2}\int_{-\infty}^{q(t)}e^z\big(u(z,t)^2-u_x(z,t)^2\big)dz.
$$
By \cite[Eq. (4.17), page 815]{liu2006} we have
\bb\label{7.0.7}
e^{-q(t)}\int_{-\infty}^{q(t)}e^z\big(u(z,t)^2-u_x(z,t)^2\big)dz\geq u(q(t),t)^2-u_x(q(t),t)^2,
\ee
that, jointly with the previous inequality, gives
$$
g'(t)\geq \f{u_x(\gamma(t))^2-u(\gamma(t))^2}{2}\geq 0.
$$

Therefore, $g$ is a non-decreasing function, and since $g(0)=u_0(x_0)-u_0'(x_0)>0$, then it is positive as long as it is defined. 

Returning to \eqref{7.0.4} and \eqref{7.0.5}, it is immediate that $f(t)=g(t)^2$, and the preceding discussion tells us that
$$
f'(t)\geq 2(u(\gamma(t))-u_x(\gamma(t)))\f{u_x(\gamma(t))^2-u(\gamma(t))^2}{2}=-f(t)^2I(t)
$$
and hence, by \eqref{7.0.2}--\eqref{7.0.3},
$$
f'(t)\geq f(t)^2(-I(t))\geq f(t)^2(-I_0)>0,
$$
meaning that 
$$
t\mapsto \f{1}{f(t)}
$$
is well defined and positive as long as $f$ exists. Therefore, we have
$$
\f{d}{dt}\Big(-\f{1}{f(t)}\Big)=\f{f'(t)}{f(t)^2}\geq -I_0.
$$

Integrating the last inequality, we get
$$
\f{1}{f(0)}-\f{1}{f(t)}\geq -I_0t,
$$
that, after rearranging terms, we obtain
$$
0<\f{1}{f(t)}-I_0t\leq \f{1}{f(0)}=\f{1}{(u_0(x_0)-u_0'(x_0))^2}.
$$

Consequently, $f(t)\rightarrow\infty$ before $t$ reaches
$$
T_0:=-\f{1}{I_0(u_0(x_0)-u_0'(x_0))^2}.
$$

Since $g_{12}(x,t)=(m(x,t)\pm2\mu)(1+\mu^2)f_{12}(x,t)$ and $g_{22}(x,t)=(1+\mu^2)f_{12}(x,t)^2$, we have
$$
g_{12}(\gamma(t))=\pm 2\mu(1+\mu^2)f(t)\quad\text{and}\quad g_{22}(\gamma(t))=(1+\mu^2)f(t)^2.
$$

Regardless $\mu$, we then have
$$
\lim_{t\nnearrow T_0}g_{22}(\gamma(t))=+\infty,
$$
while
$$
\lim_{t\nnearrow T_0}|g_{12}(\gamma(t))|=
\left\{
\ba{lcl}
+\infty, &\text{if}&\mu\neq0,\\
\\
0, &\text{if}&\mu=0.
\ea
\right.
$$
\hfill$\square$

{\bf Proof of Theorem \ref{teo2.6}.} It suffices to prove that $u$ is a solution in the class $C^{0}(H^{4}(\R),[0,\infty))\cap C^{1}(H^{3}(\R),[0,\infty))$, since it then implies that $u$ is defined on $\R\times(0,\infty)$.

Applying \cite[Theorem 1.1]{freire-AML}, with $n=4$ and $b=3$, we can guarantee that $u\in C^{0}(H^{4}(\R),[0,\infty))\cap C^{1}(H^{3}(\R),[0,\infty))$ provided that $|u_x(\cdot,t)|$ is bounded. 

For each positive integer $\ell$, let 
\bb\label{7.0.8}
v_\ell(x,t)=-\f{1}{2}\int_{-\ell}^{+\ell}\sign{(x-y)}e^{-|x-y|}m(y,t)dy.
\ee
It is immediate to check that
\bb\label{7.0.9}
|v_\ell(x,t)|\leq \f{1}{2}\int_{-\ell}^{+\ell}e^{-|x-y|}|m(y,t)|dy\leq\int_{-\infty}^{+\infty}|m(y,t)|dy=\|m(\cdot,t)\|_{L^1}.
\ee

Since $m_0$ is either non-negative or non-positive, by item $c)$ in Theorem \ref{teo3.1}, we conclude that $\sign{m(x,t)}=\sign{m_0(x)}$. In view of the conserved quantity $E_1$ in \eqref{3.0.1}, we have $E_1(t)=\|m_0\|_{L^1}$ or $E_1(t)=-\|m_0\|_{L^1}$, depending only on whether $m_0$ is non-negative or non-positive. As a result, we have 
$$
\|m(\cdot,t)\|_{L^1}=\int_\R |m(y,t)|dy=\pm\int_\R m(y,t)dy=\pm\int_\R m_0(y)dy=\int_\R|m_0(y)|dy=\|m_0\|_{L^1},
$$
that, combined with \eqref{7.0.8} and \eqref{7.0.9}, give 
\bb\label{7.0.10}
|v_\ell(x,t)|\leq \|m_0\|_{L^1}.
\ee

Finally, we observe from \eqref{7.0.8} that
$$
\ba{lcl}
v_\ell(x,t)&\rightarrow&\ds{-\f{1}{2}\int_{-\infty}^{+\infty}\sign{(x-y)}e^{-|x-y|}m(y,t)dy}\\
\\
&=&\ds{\p_x\int_{-\infty}^{+\infty}\f{e^{-|x-y|}}{2}m(y,t)dy=\p_xu(x,t)=u_x(x,t)}
\ea
$$
pointwise as $\ell\rightarrow\infty$. This, jointly with \eqref{7.0.8}--\eqref{7.0.10}, yields the result.
\hfill$\square$

\section{Pseudo-potentials and lack of geometric integrability}\label{sec8}

We now investigate whether \eqref{1.0.12} may define a one-parameter non-trivial family of PSS.

\begin{proposition}\label{prop8.1}
Let $\omega_1$, $\omega_2$ and $\omega_3$ be the one-forms \eqref{1.0.12}, $\gamma$ and $\Gamma$ functions defined by
\bb\label{8.0.1}
-2d\Gamma=\omega_3+\omega_2-2\Gamma\omega_1+\Gamma_2(\omega_3-\omega_2)
\ee
and
\bb\label{8.0.2}
2d\gamma=\omega_3-\omega_2-2\gamma\omega_1+\gamma^2(\omega_3+\omega_2).
\ee

Then system \eqref{8.0.1}--\eqref{8.0.2} is completely integrable and the one forms
\bb\label{8.0.3}
\theta=\omega_1-\Gamma(\omega_3-\omega_2)
\ee
and
\bb\label{8.0.4}
\hat{\theta}=-\omega_1+\gamma(\omega_3+\omega_2)
\ee
are closed on the solutions of the DP equation.
\end{proposition}

\begin{proof}
The proof that \eqref{8.0.1}--\eqref{8.0.2} is completely integrable under the conditions above follows from \cite[Proposition 4.1]{chern}, \cite[Theorem 2.2]{cat} or \cite[Lemma 3.2]{reyes2006-sel}. 

Applying $d$ to $\theta$ and $\hat{\theta}$ and taking \eqref{1.0.12} into account, we conclude the result.
\end{proof}

Let us consider the form $\hat{\theta}=\theta_1(x,t)dx+\theta_2(x,t)dt$ and rewrite the forms \eqref{1.0.12} as \eqref{2.3.2}. A simple calculation gives
\bb\label{8.0.5}
\ba{lcl}
\theta_1&=&\gamma(f_{31}+f_{21})-f_{11},\\
\\
\theta_2&=&\gamma(f_{32}+f_{22})-f_{12}.
\ea
\ee
Let ${\cal E}:=u_t-u_{txx}+4uu_x-3u_xu_{xx}-uu_{xxx}$. Proposition \ref{prop8.1} says that $\hat{\theta}$ is closed on the solutions of the DP equation, that is,
$$
d\hat{\theta}\big|_{{\cal E}=0}=(\p_x\theta_2-\p_t\theta_1)\big|_{{\cal E}=0}dx\wedge dt\equiv0,
$$
meaning that,
$$
\p_x\theta_2-\p_t\theta_1=0
$$
is a conservation law for the DP equation. If we assume that $\theta_2$ vanishes at $x=\pm\infty$ and
\bb\label{8.0.6}
\int_\R\theta_1(x,t)dx
\ee
is convergent for any $t$ for which $u$ is defined, then
$$
\f{d}{dt}\int_\R \theta_1(x,t)dx=\theta_2(x,t)\Big|_{-\infty}^{\infty}=0,
$$
meaning that \eqref{8.0.6} is a conserved quantity on the solutions of the DP equation.

Substituting the coefficients of the forms \eqref{1.0.12} into \eqref{8.0.5}, defining $\zeta:=-(\mu\mp\sqrt{1+\mu^2})$, taking $\zeta^{-1}=\mu\pm\sqrt{1+\mu^2}$ into account, we obtain
\bb\label{8.0.7}
\ba{lcl}
\theta_1&=&\zeta(m-2)\gamma-m,\\
\\
\theta_2&=&\ds{\Big(\f{\gamma}{\zeta}-1\Big)(u_x^2-2uu_x-uu_{xx})}.
\ea
\ee
This proves the following result.
\begin{theorem}\label{teo8.1}
    Let $\gamma$ be the function defined in Proposition \ref{prop8.1}. Then
    \bb\label{8.0.8}
    \Big(\zeta(m-2)\gamma-m\Big)_t=\Big(\Big(\f{\gamma}{\zeta}-1\Big)(u_x^2-2uu_x-uu_{xx})\Big)_x
    \ee
    defines a conservation law for the DP equation.
\end{theorem}

\begin{definition}{\tt (\cite[Definition 2.1]{reyes2006-sel})}
    An equation is said to be {\it geometrically integrable} if it describes a nontrivial one-parameter family of pseudo-spherical surfaces.
\end{definition}

The qualification {\it nontrivial one-parameter family of pseudo-spherical surfaces} has to be understood as follows: the triad of one-forms defining the first fundamental form and the Levi-Civita connection of the surface depends on a parameter that cannot be removed from any allowed transformation. That said, the pseudo-potentials also have a dependence on a parameter that could be eliminated.

\begin{theorem}\label{teo8.2}
    The DP equation, with the forms \eqref{1.0.12}, is not geometrically integrable.
\end{theorem}

\begin{proof}
    If the DP equation were integrable, theN parameter $\mu$ could not be eliminated from the one-forms \eqref{1.0.12}. Given that the functions $\mu\mapsto \mu+\sqrt{1+\mu^2}>0$ and $\mu\mapsto \mu-\sqrt{1+\mu^2}<0$ are diffeomorphisms from $\R$ into $(0,\infty)$ or $(-\infty,0)$, respectively, then parameter $\zeta$ would not be removed from \eqref{8.0.7}. Consequently, expanding $\gamma$ in an infinite series for $\zeta$, we would then obtain an infinite number of conservation laws for the DP equation. As a result, we would have an infinite number of conserved quantities for the equation.

    Given the comments above, it suffices to prove that the parameter $\zeta$ can be eliminated and \eqref{8.0.7} implies on the existence of a single conserved quantity.

    From \eqref{8.0.2} and \eqref{1.0.12} and some reckoning, we have
    \bb\label{8.0.9}
    \ba{lcl}
    2\gamma_x&=&\ds{\zeta(m-2)-2\gamma m+\f{\gamma^2}{\zeta}(m+2)},\\
    \\
    2\gamma_t&=&\ds{\Big(\zeta-2\gamma+\f{\gamma^2}{\zeta}\Big)(u_x^2-2uu_x+uu_{xx}).}
    \ea
    \ee

    Define a new function $\bar{\gamma}$ through the relation $\gamma:=\zeta(1+\bar{\gamma})$ and noticing that $\zeta\neq0$, from \eqref{8.0.9} we have
    \bb\label{8.0.10}
    \ba{lcl}
    2\bar{\gamma}_x&=&4\bar{\gamma}+\bar{\gamma}^2(m+2),\\
    \\
    2\bar{\gamma}_t&=&\bar{\gamma}^2(u_x^2-2uu_x+uu_{xx}),
    \ea
    \ee
    showing that the parameter $\zeta$ can be eliminated from \eqref{8.0.8} and thus, $\mu$ can be removed from \eqref{1.0.12} as well.
\end{proof}

The pseudo-potentials \eqref{8.0.10} were obtained from the one-forms \eqref{1.0.12}. However, reversing the process, from the pseudo-potentials \eqref{8.0.10} we can obtain a new triad of one-forms.

\begin{theorem}\label{teo8.3}
    The pseudo-potentials \eqref{8.0.10} are defined by the one-forms
    \bb\label{8.0.11}
    \ba{lcl}
    \theta_1&=&-2dx,\\
    \\
    \theta_2&=&\ds{\Big(1+\f{m}{2}\Big)dx+\f{u_x^2-2uu_x+uu_{xx}}{2}dt},\\
    \\
    \theta_3&=&\ds{\Big(1+\f{m}{2}\Big)dx+\f{u_x^2-2uu_x+uu_{xx}}{2}dt.}
    \ea
    \ee
\end{theorem}

The proof of Theorem \ref{teo8.3} is immediate and, therefore, omitted.

Consider a $\frak{sl}(2,\R)$ valued ZCR \eqref{1.0.7} for a PSS equation ${\cal E}=0$, and let $S\in SL(2,\R)$ be a matrix. Defining $\bar{X}=SXS^{-1}$ and $\bar{T}=STS^{-1}$, it is not difficult to check that
$$
\p_tX-\p_xT+[X,T]=S\Big(\p_t\bar{X}-\p_x\bar{T}+[\bar{X},\bar{T}]\Big)S^{-1},
$$
meaning that the ZCR is preserved under the Gauge transformation 
\bb\label{8.0.12}
X\mapsto SXS^{-1},\quad T\mapsto STS^{-1}.
\ee

\begin{theorem}\label{teo8.4}
The ZCR determined by \eqref{1.0.12} is equivalent to that determined by \eqref{8.0.11}.
\end{theorem}

\begin{proof}
We first assume that $\mu\neq0$. Let ZCR determined by \eqref{8.0.11} is given by the matrices
\bb\label{8.0.13}
\ba{lcl}
\bar{X}=\ds{\f{1}{2}}\begin{pmatrix}
    1+\ds{\f{m}{2}} & -3-\ds{\f{m}{2}}\\
    \\
    -1+\ds{\f{m}{2}} & -1-\ds{\f{m}{2}}
\end{pmatrix},
&&
\bar{T}=\ds{\f{f_{12}}{2}\begin{pmatrix}
    \ds{\f{1}{2}} & -\ds{\f{1}{2}}\\
    \\
    \ds{\f{1}{2}} & -\ds{\f{1}{2}}
\end{pmatrix}},
\ea
\ee
where $f_{12}=u_x^2-2uu_x+uu_{xx}$, whereas that determined by \eqref{8.0.12} is given by
\bb\label{8.0.14}
\ba{lcl}
X_\pm&=&\ds{\f{1}{2}}\begin{pmatrix}
    \mu m\pm2\sqrt{1+\mu^2} & (1\mp\sqrt{1+\mu^2})m-2\mu\\
    \\
    (1\pm\sqrt{1+\mu^2})m+2\mu & -\mu m\mp2\sqrt{1+\mu^2}
\end{pmatrix},\\
\\
T_\pm&=&\ds{\f{f_{12}}{2}}\begin{pmatrix}
    \mu & 1\mp\sqrt{1+\mu^2}\\
    \\
    1\pm\sqrt{1+\mu^2} & -\mu
\end{pmatrix}.
\ea
\ee

Let
\bb\label{8.0.15}
S_\pm=\begin{pmatrix}
    a_\pm &b_\pm\\
    \\
    c_\pm  & d_\pm
\end{pmatrix},
\ee
where
$$
\ba{lcl}
a_+= \ds{\frac{6+8 \mu ^2+\mu  \left(4-8 \sqrt{1+\mu ^2}\right)+m (-1+2 \mu ) \left(1-2 \mu +2 \sqrt{1+\mu ^2}\right)}{-2+m (-1+2 \mu )+4 \sqrt{1+\mu ^2}}},\\
\\
a_-=\ds{m (-1+2 \mu ) \left(-3+\mu +3 \sqrt{1+\mu ^2}\right)-2 \left(3+\mu +6 \mu ^2-3 \sqrt{1+\mu ^2}+2 \mu  \sqrt{1+\mu ^2}\right)},\\
\\
b_+=\ds{-1-2 \mu +2 \sqrt{1+\mu ^2}},\\
\\
b_-=\ds{ -\left(-1+3 \mu +\sqrt{1+\mu
^2}\right) \left(2+m-2 m \mu +4 \sqrt{1+\mu ^2}\right)},\\
\\
c_+=-1,\\
\\
c_-=\ds{ \left(-1-\mu +\sqrt{1+\mu ^2}\right) \left(2+m-2 m \mu +4 \sqrt{1+\mu ^2}\right)},\\
\\
d_+=1,\\
\\
d_-=\ds{m (-1+2 \mu ) \left(-1-\mu +\sqrt{1+\mu ^2}\right)+2 \left(-1+\mu
-2 \mu ^2+\sqrt{1+\mu ^2}+2 \mu  \sqrt{1+\mu ^2}\right) }.
\ea
$$

Then we have $\overline{X}=S_\pm X_\pm S_\pm^{-1}$ and $\overline{T}=S_\pm T_\pm S_\pm^{-1}$.

We now consider $\mu=0$. In this case the matrices $S_\pm$ changes to
$$
S_+=\begin{pmatrix}
3 & 1\\
\\
-1 & 1\\
\end{pmatrix}
$$
and 
$$
S_-=\begin{pmatrix}
-1 & -3\\
\\
-1 & 1
\end{pmatrix},
$$
respectively, and the result follows.
\end{proof}
\begin{remark}
    The matrices $S_\pm$ in Theorem \ref{teo8.4} are defined up to a scaling, in the sense that if $\lambda\neq0$, then replacing $S_\pm$ by $\lambda S_\pm$ we still transform $X_\pm$ and $T_\pm$ into $\overline{X}$ and $\overline{T}$, respectively, according to \eqref{8.0.12}.
\end{remark}

The fact that we can eliminate the parameter $\mu$ from the ZCR is sufficient to show that the DP cannot be geometrically integrable using the representation \eqref{1.0.12}. That said, it is quite natural to expect that the parameter $\mu$ could be eliminated from the original triad of one-forms.

\begin{theorem}\label{teo8.5}
    The two triads of one-forms \eqref{1.0.12} and \eqref{8.0.11} are related through the transformation
    $$
\begin{pmatrix}
\omega_1\\
\\
\omega_2\\
\\
\omega_3
\end{pmatrix}=\ds{
\begin{pmatrix}
1& 2 &0\\
\\
\mu\mp\sqrt{1+\mu^2}& 2\mu &0\\
\\
-\mu\pm\sqrt{1+\mu^2} & 0 &  \pm2\sqrt{1+\mu^2}
\end{pmatrix}}
\begin{pmatrix}
\theta_1\\
\\
\theta_2\\
\\
\theta_3
\end{pmatrix}.
$$
\end{theorem}

\begin{proof}
    Immediate.
\end{proof}

\begin{remark}
    Theorem \ref{teo8.2} is a corollary from both theorems \ref{teo8.4} and \ref{teo8.5}.
\end{remark}

\section{Discussion}\label{sec9}

As mentioned in the Introduction, the CH \eqref{1.0.3} and the DP \eqref{1.0.2} equations are the two most famous members of the $b-$equation \eqref{1.0.1}. While such a preeminence is due to the integrability structures they have, these properties could not be more dramatically different.

The CH equation has a second order Lax pair \eqref{1.0.5}, leading to a $\frak{sl}(2,\R)$ ZCR, that is a strong indication of geometric integrability, a fact first proved by Reyes \cite{reyes2002}. In essence, the existence of a Lax pair ensures the presence of a non-trivial parameter, that combined with Laurent type expansions can be used to derive conserved quantities for the equation. On the other hand, from the ZCR we can find one-forms satisfying the structure equations \eqref{1.0.8} for a PSS having a parameter that cannot be removed. Similarly as for the ZCR, such parameter leads to conserved quantities in an analogous process. This fact was explicitly explored by the CH equation in \cite[Section 4]{reyes2002}.

Likewise the CH equation, the DP also possesses a ZCR, but their similarity in this aspect ends here, since the DP's representation is $\frak{sl}(3,\R)$ valued. This is easily inferred from the third order Lax pair \eqref{1.0.6}. Unlike the CH equation, no one would expect the DP equation to describe PSS based on its Lax pair. In light of these comments, \cite[Example 2.8]{keti2015}, showing the one-forms \eqref{1.0.12}, is particularly disconcerting.

Despite these geometric differences, both CH and DP equations have solutions breaking in finite time as well as permanent solutions, see \cite{const1998-1,const1998-2,const2000-1,liu2006,liu2007,yin2003,yin2004}. Actually, the $b-$family (possibly with specific exceptions) has global and wave breaking solutions, e.g. see \cite{freirejde,pri-AML,freire-AML,zhou}.

The techniques employed in 
\cite{const1998-1,const1998-2,const2000-1} and
\cite{liu2006,liu2007,yin2003,yin2004} that were pioneering works in addressing the problem of blow up and global existence for the CH and DP equations, respectively, are significantly different, even more when concerning wave breaking of solutions. In some sense, this is due to the fact that one of the conserved quantities for the CH equation implies on the invariance of the $H^1(\R)$ invariance of its solutions. 

Despite it has been long known that the CH equation describes PSS and its solutions develop breaking waves \cite{const1998-1,const1998-2,const2000-1}, until quite few ago these two different aspects had not been considered together. In a recent work \cite{freire2023-1} the author proved that certain initial data leading to breaking wave solutions give rise to solutions describing PSS whose metric tensor blows up within a strip of finite height. In a subsequent work \cite{freire2023-2} it was shown that any blowing up solution of the CH equation leads to a surface with unbounded metric tensor in finite region. Actually, it was shown that the metric blows up if and only if the solution breaks in finite time. Also, problems of more global nature were also considered.

The notions of $C^k$ PSS modelled by a space function ${\cal B}$, \cite[Definition 2.1]{freire2023-1}, and generic solution, \cite[Definition 2.2]{freire2023-1}, that correspond to our definitions \ref{def2.1} and \ref{def2.2}, respectively, extended previous notions introduced in the seminal work \cite{chern}, and made possible the results in \cite{freire2023-1,freire2023-2}.

Once the machinery for considering PSS determined by Cauchy problems is built in \cite{freire2023-1,freire2023-2}, and it is known that the DP equation is also a PSS equation, then it is inevitable to investigate geometric consequences of Cauchy problems involving the DP equation in light of the recent developments mentioned.

Our Theorem \ref{teo2.2} shows that any non-trivial initial datum leads to the existence of a PSS, that is an intriguing result when considered in light with the existence of a $\frak{sl}(3,\R)$ ZCR of the DP equation. At first sight it might suggest that all solutions of the DP equation define the same surface. This possibility, however, is not supported by our results: for instance, our Theorem \ref{teo2.5} prescribes the destruction of the corresponding surface within a region of finite height, whereas theorem \ref{teo2.6} ensures the existence of well behaved and bounded coframes in $\R\times(0,\infty)$. In addition, our Theorem \ref{teo2.3} says that we cannot have any surface defined on the left side of a curve entirely determined by the initial datum of the solution. Altogether, theorems \ref{teo2.3}, \ref{teo2.5} and \ref{teo2.6} show that varying the curve $x\mapsto(x,0,u_0(x))$, $u_0\in H^4(\R)$, we obtain completely different surfaces.

Given the points above, the situation that seems to be more likely to happen is the existence of a non-trivial map (possibly, a composition of maps, eventually of non-local nature), connecting the DP with the Sine-Gordon equation. This is a topic worth of further investigation.

From the point of view of geometry, the CH and DP are also significantly different, since the former does not have any known second fundamental form depending on the independent variables and derivatives of the solution up to a finite order. For each non-trivial initial datum the solution of the Cauchy problem \eqref{3.0.3} has a second fundamental form determined everywhere over any simply connected set ${\cal C}$ such that $\omega_1\wedge\omega_2\big|_p\neq0$, for all $p\in{\cal}$, see Theorem \ref{teo2.4}.

\section{Conclusion}\label{sec10}

We prove that any curve $x\mapsto(x,0,u_0(x))$, $u_0\in H^4(\R)$, $u_0\not\equiv0$, define a strip in which a coframe for a (unique, up to a choice of sign) PSS can be defined on some open and simply connected domain. Moreover, on the same domain we also have locally defined connection forms, that jointly with the coframe, determine the first and second fundamental forms for a PSS. Depending on the behavior of the function $u_0$ (initial datum), the one-forms related to the first fundamental form can be defined everywhere, that is not the same to say that they are LI on the entire region, or they can only be defined within a region of finite height and the metric determined by then blows up.

We also showed that our original one-forms, that have a parameter dependence, can be transformed into a new set of forms without any parameter, revealing that the original parameter is not essential. 

\section*{Statements and Declarations}

The author states that no data is used to produce or support the results reported. All of them are built from mentioned literature in the text as well as the results proved within the manuscript.

The author declares there is no conflict of interests to disclose.

\section*{Acknowledgements}
I am thankful to the Department of Mathematical Sciences of the Loughborough University for the warm hospitality and amazing work atmosphere. I am particularly grateful to Jenya Ferapontov and Sasha Veselov for stimulating discussions. Particular thanks is given to Jenya Ferapontov for suggestions made on an early draft. Finally, I'd like to thank CNPq (grant no 310074/2021-5) and FAPESP (grants no 2020/02055-0 and 2022/00163-6) for financial support.

\appendix

\section{Appendix}\label{ap1}

Let us prove \eqref{7.0.3}. For simplicity, let $q:=q(t)$. The key to prove it is showing that
$$
e^{-q}\int_q^\infty e^{-x}\Big(u(x,t)^2-u_x(x,t)^2\Big)dx \geq u(q,t)^2-u_x(q,t)^2,
$$
that is quite similar to \eqref{7.0.7}. This inequality is claimed on \cite[page 816]{liu2006}, but not proved. Instead, they prove \eqref{7.0.7} and argue that \eqref{7.0.3} could be proved in a similar way. For sake of completeness, we present \eqref{7.0.3}.

To begin with, noticing that $u=g\ast m$, we have
$$
u(x,t)=\f{e^{-x}}{2}\int_{-\infty}^x e^\eta m(\eta,t)d\eta+\f{e^x}{2}\int_x^\infty e^{-\eta}m(\eta,t)d\eta
$$
Differentiating the expression above, we can obtain
\bb\label{a01}
\ba{lcl}
u(x,t)+u_x(x,t)&=&\ds{e^x\int_x^\infty e^{-\eta}m(\eta,t)d\eta}=:M(x,t)\\
\\
u(x,t)-u_x(x,t)&=&\ds{e^{-x}\int^x_{-\infty} e^{-\eta}m(\eta,t)d\eta}=:I(x,t).
\ea
\ee

Assuming the condition in Theorem \ref{teo2.5}, we conclude that
\bb\label{a02}
\ba{lcl}
M(t)&:=&M(q,t)>0,\\
\\
I(t)&:=&I(q,t)<0.
\ea
\ee

Define
$$
F(x,t):=\int_x^\infty e^{-\xi}m(\xi,t)d\xi.
$$

Integration by parts yields
\bb\label{a03}
\ba{lcl}
\ds{\int_{-\infty}^q e^{-\eta}F(\eta,t)d\eta}&=&\ds{e^{-q}\int_{-\infty}^q e^{-\xi}m(\xi,t)d\xi+\int_{-\infty}^q e^{-2\eta}m(\eta,t)d\eta}\\
\\
&=&\ds{e^{-2q}M(t)+\underbrace{\int_{-\infty}^q e^{-2\eta}m(\eta,t)d\eta}_{\geq0}\geq e^{-2q}M(t).
}
\ea
\ee

From \eqref{a01} and \eqref{a02}, we have
\bb\label{a04}
u(q,t)^2-u_x(q,t)^2=M(t)I(t)<0,
\ee
that, jointly with \eqref{a01}, gives
$$
\ba{l}
\ds{\int_q^\infty e^{-x}\Big(u(x,t)^2-u_x(x,t)^2\Big)dx}=\ds{\int_q^\infty e^{-x}M(x,t)I(x,t)dx}\\
\\
=\ds{\int_q^\infty e^{-x}\Big(e^x\int_x^\infty e^{-\eta}m(\eta,t)d\eta\Big) \Big(e^{-x}\int^x_{-\infty} e^{-\eta}m(\eta,t)d\eta\Big)dx}\\
\\
=\ds{\int_q^\infty e^{-x}\underbrace{\Big(\int_x^\infty e^{-\eta}m(\eta,t)d\eta\Big)}_{\leq0} \Big(\underbrace{\int^q_{-\infty} e^{-\eta}m(\eta,t)d\eta}_{\geq0}+\underbrace{\int^x_{q} e^{-\eta}m(\eta,t)d\eta}_{\leq0}\Big)dx}\\
\\
\geq\ds{\int_q^\infty e^{-x}\Big(\int_x^\infty e^{-\eta}m(\eta,t)d\eta\Big)\Big(\int^q_{-\infty} e^{-\eta}m(\eta,t)d\eta\Big)dx}\\
\\
=\ds{\int^q_{-\infty} e^{-\eta}m(\eta,t)d\eta\int_q^\infty e^{-x}\Big(\int_x^\infty e^{-\eta}m(\eta,t)d\eta\Big)dx.
}
\ea
$$

Substituting \eqref{a03} into the last expression and using \eqref{a01}--\eqref{a02}, we get
$$
\ba{l}
\ds{\int_q^\infty e^{-x}\Big(u(x,t)^2-u_x(x,t)^2\Big)dx}\geq \ds{\int^q_{-\infty} e^{-\eta}m(\eta,t)d\eta\int_q^\infty e^{-x}\Big(\int_x^\infty e^{-\eta}m(\eta,t)d\eta\Big)dx}\\
\\
\geq\ds{\Big(e^q I(t)\Big)\Big(e^{-2q}M(t)\Big)=e^{-q}M(t)I(t)}.
\ea
$$

Therefore, we have
\bb\label{a05}
-\f{e^q}{2}\int_q^\infty e^{-x}\Big(u(x,t)^2-u_x(x,t)^2\Big)dx\leq-\f{M(t)I(t)}{2},
\ee
that, jointly with \eqref{a04}, implies \eqref{7.0.3}.

\end{document}